\documentclass[10pt]{amsart}
\usepackage{amssymb}
\usepackage{dsfont}
\usepackage{amscd}
\usepackage[mathscr]{euscript} 
\usepackage[all]{xy}
\usepackage{url}
\usepackage{stmaryrd}
\usepackage[retainorgcmds]{IEEEtrantools}
\usepackage{hyperref}
\title{Integrating Hasse-Schmidt derivations}
\author[D. HOFFMANN]{Daniel Hoffmann$^{\dagger}$}
\thanks{2010 \textit{Mathematics Subject Classification}. Primary 13N15; Secondary 14L15.}
\thanks{\textit{Key words and phrases}. Hasse-Schmidt derivations, group scheme actions.}
\thanks{$^{\dagger}$AMDG}
\address{$^{\dagger}$Instytut Matematyczny\\
Uniwersytet Wroc{\l}awski\\
Wroc{\l}aw\\
Poland}
\email{daniel.hoffmann@math.uni.wroc.pl}
\author[P. KOWALSKI]{Piotr Kowalski$^{\spadesuit}$}
\thanks{$^{\spadesuit}$Supported by NCN grant 2012/07/B/ST1/03513}
\address{$^{\spadesuit}$Instytut Matematyczny\\
Uniwersytet Wroc{\l}awski\\
Wroc{\l}aw\\
Poland}
 \email{pkowa@math.uni.wroc.pl} \urladdr{http://www.math.uni.wroc.pl/\textasciitilde pkowa/ }

  \DeclareMathOperator{\id}{id}
 \DeclareMathOperator{\fr}{Fr} \DeclareMathOperator{\lie}{Lie}
  
\DeclareMathOperator{\im}{im}  
\DeclareMathOperator{\ch}{char}

\DeclareMathOperator{\ev}{ev}

\DeclareMathOperator{\evv}{ev}

\newtheorem{theorem}{Theorem}[section]
\newtheorem{prop}[theorem]{Proposition}
\newtheorem{lemma}[theorem]{Lemma}
\newtheorem{cor}[theorem]{Corollary}

\theoremstyle{definition}
\newtheorem{definition}[theorem]{Definition}
\newtheorem{example}[theorem]{Example}
\newtheorem{remark}[theorem]{Remark}
\newtheorem{question}[theorem]{Question}

\begin{document}

\newcommand{\twoc}[3]{ {#1} \choose {{#2}|{#3}}}
\newcommand{\thrc}[4]{ {#1} \choose {{#2}|{#3}|{#4}}}
\newcommand{\Zz}{{\mathds{Z}}}
\newcommand{\Ff}{{\mathds{F}}}
\newcommand{\Cc}{{\mathds{C}}}
\newcommand{\Rr}{{\mathds{R}}}
\newcommand{\Nn}{{\mathds{N}}}
\newcommand{\Qq}{{\mathds{Q}}}
\newcommand{\Kk}{{\mathds{K}}}
\newcommand{\Pp}{{\mathds{P}}}
\newcommand{\ddd}{\mathrm{d}}
\newcommand{\Aa}{\mathds{A}}
\newcommand{\dlog}{\mathrm{ld}}
\newcommand{\ga}{\mathbb{G}_{\rm{a}}}
\newcommand{\gm}{\mathbb{G}_{\rm{m}}}
\newcommand{\gaf}{\widehat{\mathbb{G}}_{\rm{a}}}
\newcommand{\gmf}{\widehat{\mathbb{G}}_{\rm{m}}}

\maketitle
\begin{abstract}

We study integrating (that is expanding to a Hasse-Schmidt derivation) derivations, and more generally truncated Hasse-Schmidt derivations, satisfying iterativity conditions given by formal group laws. Our results concern the cases of the additive and the multiplicative group laws. We generalize a theorem of Matsumura about integrating nilpotent derivations (such a generalization is implicit in work of Ziegler) and we also generalize a theorem of Tyc about integrating idempotent derivations.
\end{abstract}

\section{Introduction}
The algebraic theory of derivations is an important tool in commutative algebra. This theory works better in the characteristic $0$ case due to the fact that any derivation vanishes on the set of $p$-th powers, if the characteristic is $p>0$. To deal with this problem, Hasse and Schmidt introduced \emph{higher differentiations} \cite[p. 224]{HS}
which we call \emph{HS-derivations} in this paper. An HS-derivation (see Definition \ref{defhs}) is a sequence of maps having a usual derivation as its first element. In the case of a $\Qq$-algebra, any derivation uniquely expands to an \emph{iterative} HS-derivation (see Definition \ref{defiter}) and the two theories coincide.

Matsumura obtained several interesting results about expanding (called \emph{integrating} in \cite{Mats1}) derivations to HS-derivations in the case of positive characteristic. The first result \cite[Theorem 6 and Corollary]{Mats1} says that any derivation on a field is (non-uniquely) integrable.
The second result \cite[Theorem 7]{Mats1} says that a derivation $D$ on a field can be integrated to an iterative HS-derivation (\emph{strongly integrated} in Matsumura's terminology) if and only if composing $D$ with itself $p$ times gives the $0$-map.
In our paper we prove a version of the second Matsumura's result mentioned above.

Analyzing the definition of an iterative HS-derivation, one realizes \cite[(1.9) and (1.10)]{Mats1} that the iterativity condition is given by the additive \emph{formal group law} $\gaf=X+Y$. It is natural to ask what happens if the additive formal group law is replaced with another formal group law $F$, e.g. the multiplicative one $\gmf=X+Y+XY$. Such a replacement naturally leads to a definition of an \emph{$F$-iterative} HS-derivation, see Definition \ref{defiter}. Such derivations were considered before in a number of contexts. In \cite{Tyc}, they are defined as actions of a formal group law on an algebra (see also \cite{cr98} and \cite{cr2001}). In a more general setting (i.e. on arbitrary schemes), they appear in \cite[Def. 4.2]{Gil} (for an even more general setting, see \cite{MS}).

Any formal group (law) $F$ can be considered as a direct limit of certain finite or truncated group laws $F[m]$. These truncated group laws give rise to the definition of $F[m]$-iterative truncated HS-derivations (see Definition \ref{deffitertrun}) and we can ask whether such derivations are integrable. We prove the following.
\begin{theorem}\label{multmats}
Let $F$ be a formal group law such that $F\cong \gaf$ or $F\cong \gmf$ and $m>0$. Then any $F[m]$-iterative truncated HS-derivations on a field of positive characteristic can be integrated to an $F$-iterative HS-derivation.
\end{theorem}
\noindent
For $m=1$ and $F\cong \gaf$, the result above is \cite[Theorem 7]{Mats1}. For $m=1$ and $F\cong \gmf$, the result above (formulated in terms of restricted Lie algebra actions) is contained in the main theorem of \cite{Tyc} (we were unaware of it while writing the first version of this paper). Similar results in the $m=1$ case were also obtained in \cite{cr98} and \cite{cr2001}.

In the additive case (corresponding to the usual iterativity) such a generalization is implicit in the work of Ziegler (\cite{Zieg3}, \cite{Zieg2}) and it was crucial to axiomatize the class of \emph{existentially closed} (in the sense of logic, see \cite{Ho}) iterative HS-fields \cite{Zieg2}. The second author also used similar ideas to find \emph{geometric} axioms of this class \cite{K3}. Our main motivation for considering this topic was to study other theories of existentially closed fields with HS-derivations, such issues are treated in \cite{HK}.

Integrating HS-derivations is also related with singularity theory via Matsumura's positive characteristic version of the Zariski-Lipman conjecture (see \cite[page 241]{Mats1}), but we do not pursue this direction here.

The paper is organized as follows. In Section \ref{secdefinitions}, we set our notation and introduce iterative HS-derivations, where the iterativity notion comes from a truncated/formal group law. In Section \ref{secalgprop}, we prove some preparatory results about composing HS-derivations and their fields of constants. In section \ref{secexpanding}, we prove our theorems about integration of truncated HS-derivations. In Section \ref{multi}, we comment on the ``partial'' (i.e. several HS-derivations) case.

We are grateful to the referee for a very useful report.

\section{Definitions and Notation}\label{secdefinitions}
\noindent
Let us fix a field $k$ and a $k$-algebra $R$. For any function $f:R\to R$, $r\in R$ and a positive integer $n$, $f^n(r)=f(r)^n$ and $f^{(n)}$ is the composition of $f$ with itself $n$ times.

\subsection{Formal group laws and HS-derivations}
\begin{definition}\label{defhs}
A sequence $\partial = (\partial_{n}:R\to R)_{n\in \Nn}$ of additive maps is
called an \emph{HS-derivation} if $\partial_{0}$ is the identity map, and for all $n\in \Nn$ and $x,y\in R$,
$$\partial_{n}(xy) = \sum_{i+j=n}\partial_{i}(x)\partial_{j}(y).$$
If moreover for all $n>0$ and $x\in k$ we have $\partial_n(x)=0$, then we call $\partial$ an \emph{HS-derivation over $k$}.
\end{definition}
\noindent
For any sequence of maps $\partial = (\partial_{n}:R\to R)_{n\in \Nn}$ such that $\partial_{0}$ is the identity map and a variable $X$, we define a map
$$\partial_X:R\to R\llbracket X\rrbracket,\ \ \partial_X(r)=\sum_{n=0}^{\infty}\partial_n(r)X^n.$$
It is easy to see \cite[p. 207]{mat} that $\partial$ is an HS-derivation if and only if $\partial_X$ is a ring homomorphism and that $\partial$ is an HS-derivation over $k$ if and only if $\partial_X$ is a $k$-algebra homomorphism. It is also clear that for a ring homomorphism $\varphi:R\to R\llbracket X\rrbracket$, $\varphi$ is of the form $\partial_X$ for some derivation $\partial$ if and only if the composition of $\varphi$ with the natural projection map $R\llbracket X\rrbracket\to R$ is the identity map.
\begin{definition}\label{defiter}
An HS-derivation $\partial$ is called \emph{iterative} if for all $i,j\in \Nn$ we have
$$\partial_{i}\circ\partial_{j} = {i+j \choose i}\partial_{i+j}.$$
\end{definition}
\noindent
Assume that $S$ is a complete local $R$-algebra and $s_1,\ldots,s_k$ belong to the maximal ideal of $S$. There is a unique $R$-algebra homomorphism
$R\llbracket X_1,\ldots,X_k\rrbracket\to S$ such that each $X_i$ is mapped to $s_i$ \cite[Theorem 7.16]{comm}. We denote this homomorphism by $\ev_{(s_1,\ldots,s_k)}$ and for $F\in R\llbracket X_1,\ldots,X_k\rrbracket$ we often write $F(s_1,\ldots,s_k)$ instead of $\ev_{(s_1,\ldots,s_k)}(F)$.
\\
It is again easy to see \cite[p. 209]{mat} that an HS-derivation $\partial$ is iterative if and only if the following diagram is commutative
\begin{equation*}
 \xymatrix{R \ar[rr]^{\partial_Y} \ar[d]_{\partial_Z} & & R\llbracket Y\rrbracket \ar[d]^{\partial_X\llbracket Y\rrbracket} \\
	   R\llbracket Z\rrbracket \ar[rr]^{\ev_{X+Y}} & & R\llbracket X,Y\rrbracket .}
\end{equation*}
\noindent
One could replace the power series $X+Y$ in the definition above with an arbitrary power series in two variables. But it turns out that we get a meaningful definition only in the case when $F\in k\llbracket X,Y\rrbracket$ is a \emph{formal group law} (over $k$) i.e. if it satisfies
$$F(X,0)=X=F(0,X),\ \ \ F(F(X,Y),Z)=F(X,F(Y,Z)).$$
\begin{remark}\label{formalnecessity}
The necessity of the first condition is clear. The necessity of the associativity condition comes from the associativity of the composition of functions together with some diagram chasing, similarly as in the proof of Proposition \ref{composing}.
\end{remark}
\begin{example} We give examples of formal group laws.
\begin{itemize}
\item The additive formal group law $\gaf=X+Y$.

\item The multiplicative formal group law $\gmf=X+Y+XY$.

\item More generally, any one-dimensional algebraic group $G$ over $k$ gives (after choosing a local parameter at $1\in G(k)$) a formal group law called the \emph{formalization} of $G$ (see \cite[Section 2.2]{Manin}) and denoted by $\widehat{G}$.

\item For any $n>0$, there is a formal group law $\bar{F}_{\Delta_n}$ (see \cite[3.2.3]{Hazew}) over $\Ff_p$. If $n>2$, then this formal group law does not come from the formalization of an algebraic group.
\end{itemize}
\end{example}

\noindent
Let us fix $F$, a formal group law over $k$.
\begin{definition}\label{defefiter}
An HS-derivation $\partial$ over $k$ is called \emph{$F$-iterative} if the following diagram is commutative
\begin{equation*}
 \xymatrix{R \ar[rr]^{\partial_Y} \ar[d]_{\partial_Z} & & R\llbracket Y\rrbracket \ar[d]^{\partial_X\llbracket Y\rrbracket} \\
	   R\llbracket Z\rrbracket \ar[rr]^{\ev_{F}} & & R\llbracket X,Y\rrbracket .}
\end{equation*}
\end{definition}
\noindent
We shorten the long phrase ``$F$-iterative HS-derivation over $k$'' to \emph{$F$-derivation}.
\\
For a one-dimensional algebraic group $G$ over $k$ we use the term \emph{$G$-derivation} rather than $\widehat{G}$-derivation. In particular, a $\ga$-derivation is the same as an iterative HS-derivation.
\\
If $F'$ is also a formal group law and $t$ is a variable, then $\alpha\in tk\llbracket t\rrbracket$ is called a homomorphism between $F$ and $F'$, denoted $\alpha:F\to F'$, if
$$\alpha(F(X,Y))=F'(\alpha(X),\alpha(Y)).$$
\noindent
The two statements below connect homomorphisms of formal group laws with iterative derivations. Note that homomorphisms on the formal group level go the \emph{opposite} direction to $k$-algebra homomorphisms (since the category of complete Hopf algebras is opposite to the category of formal groups, see Section \ref{gpscheme}).
\begin{lemma}\label{pullpushf}
Assume that $\alpha:F\to F'$ is a homomorphism of formal group laws over $k$ and we have the following commutative diagram
\begin{equation*}
  \xymatrix{
R\llbracket X\rrbracket  \ar[rr]^{\ev_{\alpha}} &  & R\llbracket X\rrbracket\\
 &   R\ar[ul]^{(\partial')_X} \ar[ur]_{\partial_X}. & }
\end{equation*}
\begin{enumerate}
\item If $\partial'$ is an $F'$-derivation, then $\partial$ is an $F$-derivation.

\item If $\partial$ is an $F$-derivation and $\ev_{\alpha}$ is one-to-one (equivalently, $\alpha\neq 0$), then $\partial'$ is an $F'$-derivation.
\end{enumerate}
\end{lemma}
\begin{proof} This is a relatively easy diagram chase which we leave to the reader.
\end{proof}

\begin{remark}
Assume that $\ch(k)=0$. Then any derivation $D$ on $R$ uniquely expands to a $\ga$-derivation $(D^{(n)}/n!)_{n\in \Nn}$. Therefore the theory of derivations coincides with the theory of iterative HS-derivations. Considering other formal group laws does not change this theory either, since by \cite[Theorem 1.6.2]{Hazew} each formal group law $F$ over $k$ is isomorphic to $\gaf$, and such an isomorphism gives a bijective correspondence (see  Lemma \ref{pullpushf}) between $\ga$-derivations and $F$-derivations.
\end{remark}
\noindent
Therefore, from now on we assume that $\ch(k)=p>0$, but sometimes we will make comments regarding the characteristic $0$ case.

\subsection{Truncated HS-derivations}
Let us fix  a natural number $m>0$.
\begin{definition}\label{defhstr}
A sequence $\partial = (\partial_{n}:R\to R)_{n<p^m}$ of additive maps is
called an \emph{$m$-truncated HS-derivation} if $\partial_{0}$ is the identity, and for all $n<p^m$ and $x,y\in R$,
$$\partial_{n}(xy) = \sum_{i+j=n}\partial_{i}(x)\partial_{j}(y).$$
If moreover for all $0<n<p^m$ and $x\in k$ we have $\partial_n(x)=0$, then we call $\partial$ an \emph{$m$-truncated HS-derivation over $k$}.
\end{definition}
\noindent
Let $v_m,w_m,u_m$ (or just $v,w,u$ if $m$ is clear from the context) denote the ``$m$-truncated variables'' e.g.
$$R[v_m]=R[X]/(X^{p^m}),\ \ R[v_m,w_m,u_m]=R[X,Y,Z]/(X^{p^m},Y^{p^m},Z^{p^m}).$$
For any sequence of maps $\partial = (\partial_{n}:R\to R)_{n<p^m}$ such that $\partial_{0}$ is the identity map, we define a map
$$\partial_{v_m}:R\to R[v_m],\ \ \partial_{v_m}(r)=\sum_{n=0}^{p^m-1}\partial_n(r)v_m^n.$$
It is easy to see again that $\partial$ is an $m$-truncated HS-derivation if and only if $\partial_{v_m}$ is a ring homomorphism and that $\partial$ is an HS-derivation over $k$ if and only if $\partial_{v_m}$ is a $k$-algebra homomorphism.
\\
\\
Since for all $i,j<p^m$ if $i+j\geqslant p^m$ then ${i+j \choose i}=0$, we can define \emph{iterative $m$-truncated HS-derivations} exactly as in Definition \ref{defiter}.
\begin{remark}\label{oneiterat}
If $\partial=(\partial_i)_{i<p^m}$ is an iterative $m$-truncated HS-derivation, then $\partial_1$ is a derivation such that $\partial_1^{(p)}=0$ (see \cite[page 209]{mat} or a more general Remark \ref{gapower}). Conversely, if $D$ is a derivation such that $D^{(p)}=0$, then $(D^{(i)}/i!)_{i<p}$ is an iterative $1$-truncated HS-derivation. Thus there is a natural bijective correspondence between the set of derivations $D$ such that $D^{(p)}=0$ and the set of iterative $1$-truncated HS-derivations.
\end{remark}
\noindent
The notion of iterativity in the truncated case can be explained again in terms of diagrams. Before this explanation we need some comments about the truncated evaluation maps. If $S$ is any $R$ algebra and $s\in S$ such that $s^{p^m}=0$, then there is a unique $R$-algebra map $\ev_s:R[v_m]\to S$ such that $\ev_s(v_m)=s$. Again, for $f\in R[v_m]$ we sometimes write $f(s)$ for $\ev_s(f)$. Similarly in the case of several truncated variables. An $m$-truncated HS-derivation $\partial$ is iterative if and only if the following diagram is commutative
\begin{equation*}
 \xymatrix{R \ar[rr]^{\partial_w} \ar[d]_{\partial_u} & & R[w] \ar[d]^{\partial_v[w]} \\
	   R[u] \ar[rr]^{\ev_{v+w}} & & R[v,w],}
\end{equation*}
where $v=v_m,w=w_m,u=u_m$ (note that $(v+w)^{p^m}=0$).
\\
\\
Again, we can replace $v+w$ with any $m$-truncated polynomial $f\in k[v,w]$, but to get a meaningful definition (see Remark \ref{formalnecessity}) we need $f$ to be an \emph{$m$-truncated group law} i.e. it should satisfy
$$f(v,0)=v=f(0,v),\ \ \ f(f(v,w),u)=f(v,f(w,u)).$$
\begin{remark}\label{trunremark}
\begin{enumerate}
\item Note that the condition $f(v,0)=v=f(0,v)$ is equivalent to saying that $f$ belongs to the maximal ideal of $k[v,w]$, which in turn is equivalent to the condition $f^{p^m}=0$. Therefore, for a truncated group law $f$, the evaluation map $\evv_f$ makes sense.

\item Truncated group laws correspond to \emph{Hopf algebra} (see e.g. \cite[Section 1.4]{Water}) structures on $k[v]$, where the comultiplication is the map $\ev_f$ (for a given truncated group law $f$). By a theorem of Cartier \cite[Section 11.4]{Water}, a Hopf algebra over a field of characteristic zero is reduced (i.e. has no non-zero nilpotent elements).

\item There is a classical notion of a \emph{(commutative) one dimensional formal group chunk of order $n$} (see \cite[Def. 5.7.1]{Hazew}), which in our positive characteristic case coincides with the notion of an $m$-truncated group law for $n=p^m$ and, as in $(2)$ above, a Hopf algebra structure on $k[v]$. However in the case of characteristic $0$, these notions differ: there are plenty of formal group chunks, but (see $(2)$ above) no Hopf algebra structures on the ring of the form $k[X]/(X^n)$ for $n>1$.
\end{enumerate}
\end{remark}
\noindent
Let us fix $f$, an $m$-truncated formal group law.
\begin{definition}\label{deffitertrun}
An $m$-iterative HS-derivation $\partial$ over $k$ is called \emph{$f$-iterative} if the following diagram is commutative
\begin{equation*}
 \xymatrix{R \ar[rr]^{\partial_w} \ar[d]_{\partial_u} & & R[w] \ar[d]^{\partial_v[w]} \\
	   R[u] \ar[rr]^{\ev_f} & & R[v,w]}
\end{equation*}
where $v=v_m,w=w_m,u=u_m$.
\end{definition}
\noindent
We shorten the long phrase ``$f$-iterative $m$-truncated HS-derivation over $k$'' to \emph{$f$-derivation}.
\\
Let $l$ be a positive integer and let us also take $g$, an $l$-truncated formal group law. We define a homomorphism between truncated group laws $f$ and $g$, denoted by $\alpha:f\to g$, as an element $\alpha\in k[v_m]$ satisfying
$$\alpha^{p^l}=0,\ \ \ \alpha(f(v_m,w_m))=g(\alpha(v_m),\alpha(w_m)).$$
\begin{remark}\label{pullpusht}
Lemma \ref{pullpushf} remains true if we replace formal group laws by truncated group laws (of course the injectivity is not equivalent anymore to being non-zero).
\end{remark}

\subsection{Truncating HS-derivations and Frobenius}\label{sectruncating}
We recall that $F$ is a formal group law over $k$ and $m>0$. Let us define
$$F[m]:= F(v_m,w_m)\in k[v_m,w_m].$$
It is clear that (see also \cite[Lemma 1.1]{Manin}) $F[m]$ is an $m$-truncated group law.
\\
For any 1-dimensional algebraic group $G$ over $k$, we use the term ``$G[m]$-derivation'' rather than ``$\widehat{G}[m]$-derivation''. In particular, we sometimes say $\ga[m]$-derivations for iterative $m$-truncated HS-derivations.
\\
Assume that $f$ is an $m$-truncated group law and take $1\leqslant l\leqslant m$. We define
$$f[l]:= f(v_l,w_l)\in k[v_l,w_l].$$
Then $f[l]$ is an $l$-truncated group law. Clearly this construction coincides with the previous one i.e.
$$F[m][l]=F[l].$$
We also have a natural homomorphism $v_l:f[l]\to f$ of truncated group laws. In particular we can understand $F$ as the limit of the direct system of truncated group schemes $(F[l])_{l>0}$ (see \cite[Lemma 1.1]{Manin}).
\begin{remark}
The situation is very different in the characteristic $0$ case where a formal group law can not be approximated in a similar fashion (see Remark \ref{trunremark}(3)).
\end{remark}
\noindent
Let $j:=m-l$. Then $v_m^{p^j}$ need \emph{not} be a homomorphism $f\to f[l]$ of truncated group laws unless $f$ is defined over the prime field $\Ff_p$. For any field automorphism $\varphi:k\to k$ and $g=\sum_{n<p^m}c_nv_m^n\in k[v_m]$, we denote by $g^{\varphi}$ the truncated polynomial $\sum_{n<p^m}\varphi(c_n)v_m^n$. Similarly for power series. Let $\fr$ denote the Frobenius automorphism of $k$. Then $g$ is a truncated group law if and only if $g^{\varphi}$ is (similarly for formal group laws). Therefore $f^{\fr^{-j}}$ is an $m$-truncated group law and $v_m^{p^j}:f^{\fr^{-j}}\to f[l]$ (similarly for formal group laws).

\section{Algebraic properties of $F$-derivations}\label{secalgprop}
\noindent%
 We assume that $k$ is a perfect field of positive characteristic $p$ and $R$ is a $k$-algebra. We also fix a formal group law $F$ over $k$, a positive integer $m$ and an $m$-truncated group law $f$.

\subsection{Morphisms of group laws and HS-derivations}\label{secmorgl}
As any formal group law over a field is necessarily commutative \cite[Theorem 1.6.7]{Hazew}, an $F$-derivation $\partial=(\partial_i)_{i\in \Nn}$ satisfies $\partial_i\circ \partial_j=\partial_j\circ \partial_i$ for all $i,j\in \Nn$. We will use this fact often.
\begin{remark}\label{grouphs}
There are also consequences of the existence of the ``inverse map'' i.e. a power series $W$ such that $F(X,W(X))=0$. There is a group operation on the set of HS-derivations on $R$ over $k$ defined as follows
$$(\partial\ast \partial')_n=\sum_{i+j=n}\partial_i\circ \partial'_j$$
(see \cite[page 208]{Mats1}). For an $F$-derivation $\partial$, the inverse of $\partial$ (with respect to the group operation above) can be expressed in terms of $W$, e.g. if $\partial$ is a $\ga$-derivation, then the inverse of $\partial$ coincides with $((-1)^n\partial_n)_{n\in \Nn}$.
\end{remark}
\begin{lemma}\label{hsfrob}
Assume that $0<l\leqslant m$ and $j=m-l$.
\begin{enumerate}
\item If $\partial$ is an $F$-derivation on $R$, then $\partial'=(\partial_i)_{i<p^m}$ is an $F[m]$-derivation.

\item If $\partial$ is an $f$-derivation on $R$, then $\partial'=(\partial_i)_{i<p^l}$ is an $f[l]$-derivation.

\item Assume that $\partial$ is an $F$-derivation on $R$, such that for any $n\in \Nn$ nondivisible by $p^m$, we have $\partial_n=0$. Let  $F'=F^{\fr^{m}}$ and $\partial'=(\partial_{ip^m})_{i\in \Nn}$. Then $\partial'$ is an $F'$-derivation.

\item Assume that $\partial$ is an $f$-derivation on $R$, such that for any $n<p^m$ nondivisible by $p^j$, we have $\partial_n=0$. Let  $f'=f^{\fr^{j}}$ and $\partial'=(\partial_{ip^j})_{i<p^l}$. Then $\partial'$ is an $f'[l]$-derivation.
\end{enumerate}
\end{lemma}
\begin{proof}
The proofs of $(1)$ and $(2)$ are straightforward.
\\
For the proof of $(3)$ consider the following commutative diagram
\begin{equation*}
  \xymatrix{
R\llbracket X\rrbracket  \ar[rr]^{\ev_{X^{p^m}}} &  & R\llbracket X\rrbracket\\
 &   R\ar[ul]^{(\partial')_{X}} \ar[ur]_{\partial_{X}}. & }
\end{equation*}
By the last paragraph of Section \ref{sectruncating}, $X^{p^m}:F\to F'$ is a homomorphism of formal group laws and clearly $\ev_{X^{p^m}}$ is one-to-one. By Lemma \ref{pullpushf}, $\partial'$ is an $F'$-derivation.
\\
The proof of $(4)$ is analogous to the proof of $(3)$ (using Remark \ref{pullpusht}), since the homomorphism
$$\ev_{v_m^{p^j}}:k[v_l]\to k[v_m]$$
is injective.
\end{proof}
\begin{remark}\label{lasthom}
\begin{enumerate}
\item This paper is concerned with the \emph{opposite} operation to the one appearing in Lemma \ref{hsfrob}(1): we take an $F[m]$-derivation and we aim to integrate (i.e. expand) it to an $F$-derivation.

\item One may ask whether any $m$-truncated group law $f$ can be integrated i.e. whether there is a formal group law $F$ such that $F[m]=f$. The answer is positive if and only if $f$ is commutative \cite[Corollary 5.7.4]{Hazew} (see \cite[Example 5.7.8]{Hazew} for an example of a non-commutative $f$).

\item We will see (Corollary \ref{pvanish}) that the vanishing condition in Lemma \ref{hsfrob}(3) above is equivalent to the vanishing of $\partial_{p^i}$ for $i=0,\ldots,m-1$. Similarly in Lemma \ref{hsfrob}(4).
\end{enumerate}
\end{remark}

\subsection{Canonical $F$-derivation}
One could wonder whether non-zero $F$-derivations exist at all. The next result provides a canonical example. Let $t$ denote a variable which will play a somewhat different role than the variables $X,Y,Z$.
\begin{prop}\label{canonicalder}
Consider the map
$$\ev_{F(t,X)}:R\llbracket t\rrbracket\to R\llbracket t,X\rrbracket.$$
There is an $F$-derivation $\partial$ on $R\llbracket t\rrbracket$ over $R$ such that $\partial_X=\ev_{F(t,X)}$.
\end{prop}
\begin{proof}
The composition of $\ev_{F(t,X)}$ with the projection map $R\llbracket t,X\rrbracket\to R\llbracket t\rrbracket$ is exactly the map $\ev_{F(t,0)}$. Since $F(t,0)=t$, this composition is the identity map. Thus there is an HS-derivation $\partial$ over $R$ on $R\llbracket t\rrbracket$ such that $\partial_X=\ev_{F(t,X)}$. It remains to check the $F$-iterativity condition, i.e. the commutativity of the following diagram (see Definition \ref{defefiter})
\begin{equation*}
 \xymatrix{R\llbracket t\rrbracket \ar[rr]^{\partial_Y} \ar[d]_{\partial_X} & & R\llbracket t, Y\rrbracket \ar[d]^{\partial_X\llbracket Y\rrbracket} \\
	   R\llbracket t,X\rrbracket \ar[rr]^{\ev_{F}} & & R\llbracket t,X,Y\rrbracket .}
\end{equation*}
Note that $\ev_{F}$ is an $R\llbracket t\rrbracket$-algebra map. We will interpret all the maps in the diagram above as evaluation maps on power series $R$-algebras. We have:
$$\partial_X=\ev_{F(t,X)},\ \partial_Y=\ev_{F(t,Y)},\ \partial_X\llbracket Y\rrbracket=\ev_{(F(t,X),Y)},\ \ev_{F}=\ev_{(t,F(X,Y))}.$$
Therefore we obtain
\begin{equation*}
 \partial_X\llbracket Y\rrbracket\circ\partial_Y=\ev_{F(F(t,X),Y)},\ \
 \ev_{F}\circ\partial_X=\ev_{F(t,F(X,Y))}.
\end{equation*}
Hence the commutativity of the diagram above is equivalent to the associativity axiom for the formal group law $F$.
\end{proof}
\begin{example}\label{gaitex}
For $F=\gaf=X+Y$ and $\partial$ the canonical $F$-derivation on $k\llbracket t\rrbracket$, we see that $\partial_1=\frac{\ddd}{\ddd t}$ is the usual derivative with respect to the variable $t$ and for any $n\in \Nn$ we have
$$ \partial_n\Big(\sum\limits_{i=0}^{\infty}a_i t^i\Big) = \sum\limits_{i=0}^{\infty}a_{i+n}{i+n\choose n}t^{i}.$$
The formula above appears already in \cite{Hasse} (last line on p. 50). It is clear that $\partial$ restricts to the polynomial ring $k[t]$.
\end{example}
\begin{example}\label{canderm}
For $F=\gmf=X+Y+XY$ and $\partial$ the canonical $F$-derivation on $k\llbracket t\rrbracket$, one can compute (similarly as in Example \ref{multiter}) that
$$\partial_n\Big(\sum\limits_{i=0}^{\infty}a_i t^i\Big)=\sum\limits_{m=0}^{\infty}\sum\limits_{j=0}^{\min(m, n)} \frac{(m+n-j)!}{(m-j)! (n-j)! j!}
a_{m+n-j}t^m.$$
Clearly, this HS-derivation restricts to the polynomial ring $k[t]$ as well.
\end{example}

\subsection{Composing HS-derivations}
In this subsection we describe a passage from the diagram describing $F$-iterativity (Definition \ref{defefiter}) to formulas for the actual composition of terms of an  $F$-derivation. The $p$-th compositional power will play a special role.
\\
\\
Let  $\partial=(\partial_0,\partial_1,\ldots)$ be an HS-derivation on $R$. Till Proposition \ref{composing}, we do not assume that $\partial$ obeys any iterativity law, although we still assume that the maps $\partial_i,\partial_j$ commute with each other (see the beginning of Section \ref{secmorgl}). For each $i\geqslant 1$, let
$$E_i:R\llbracket X_1,\ldots,X_{i-1}\rrbracket\to R\llbracket X_1,\ldots,X_{i}\rrbracket,\ \ E_i=\partial_{X_i}\llbracket X_1,\ldots,X_{i-1}\rrbracket.$$
In particular $E_1=\partial_{X_1}$. For any $m\geqslant 1$, let $E_{(m)}$ denote the composition of the maps below:
\begin{equation*}
  \xymatrix{R \ar[r]^{E_1\ \ \ }& R\llbracket X_1\rrbracket\ar[r]^{E_2\ \ }& R\llbracket X_1,X_2\rrbracket \ar[r]^{\ \ \ \ E_3}& \ldots \ar[r]^{E_{m}\ \ \ \ \ \ \ \ } &
  R\llbracket X_1,\ldots,X_{m}\rrbracket  \ar[rr]^{\ \ \ \ \ \ev_{(X,\ldots,X)}} & &R\llbracket X\rrbracket.}
\end{equation*}
Then the map $E_{(m)}:R\to R\llbracket X\rrbracket$ coincides with $(\partial^{\ast m})_X$, where $\partial^{\ast m}=\partial\ast \ldots \ast \partial$ ($m$ times) for the group operation $\ast$ on the set of HS-derivations on $R$ from Remark \ref{grouphs}.
\begin{lemma}\label{epe}
For any $r\in R$ we have:
$$E_{(p)}(r)=\sum_{i=0}^{\infty}\partial_i^{(p)}(r)X^{pi}.$$
\end{lemma}
\begin{proof} We have
$$E_{(p)}(r)=\sum_{n=0}^{\infty}\sum_{i_1+\ldots+i_p=n}(\partial_{i_1}\circ \ldots \circ \partial_{i_p})(r)X^{i}.$$
Since all the maps $\partial_i$ commute with each other, the lemma follows from the fact that the number of different permutations of the sequence $(i_1,\ldots,i_p)$ is one if $i_1=\ldots=i_p$ and divisible by $p$ otherwise.
\end{proof}
\begin{cor}\label{ppowerder}
Under the assumptions above, $\partial^{(p)}:=(\partial_i^{(p)})_i$ is an HS-derivation.
\end{cor}
\begin{proof} Let $\iota:R\llbracket X^p\rrbracket\to R\llbracket X\rrbracket$ denote the inclusion map. By Lemma \ref{epe}, we have
$$\iota\circ (\partial^{(p)})_{X^p}=E_{(p)}.$$
Since $E_{(p)}$ is a ring homomorphism, $(\partial^{(p)})_{X^p}$ is a ring homomorphism as well, so $\partial^{(p)}$ is an HS-derivation.
\end{proof}
\noindent
For an $F$-derivation $\partial$, we aim to express $\partial^{(p)}$ in terms of $F$ and $\partial$. For any positive integer $m$, let
$$[m+1]_F(X)=F(X,[m]_F(X))$$
be the ``multiplication by $m$ map'' ($[1]_F=X$).
\begin{prop}\label{composing}
If $\partial$ is an $F$-derivation, then the following diagram commutes:
\begin{equation*}
  \xymatrix{
R  \ar[rr]^{E_{(m)}} \ar[rd]_{\partial_X} &  & R\llbracket X\rrbracket\\
 &  R\llbracket X\rrbracket .\ar[ru]_{\evv_{[m]_F}} & }
\end{equation*}
\end{prop}
\begin{proof}
Let $F_2$ denote $F(X_1,X_2)$ and for $m\geqslant 2$ let
$$F_{m+1}=F_m(X_1,\ldots,X_{m-1},F(X_m,X_{m+1})).$$
By the definitions of $E_{(m)}$ and $[m]_F$, it is enough to prove the following.
\\
{\bf Claim}
\\
For each $m\geqslant 2$, the following diagram is commutative
\begin{equation*}
  \xymatrix{R \ar[r]^{E_1\ \ \ } \ar[rrd]_{E_1}& R\llbracket X_1\rrbracket\ar[r]^{E_2\ \ }& R\llbracket X_1,X_2\rrbracket \ar[r]^{\ \ \ \ E_3}& \ldots \ar[r]^{E_{m}\ \ \ \ \ \ \ \ } &
  R\llbracket X_1,\ldots,X_{m}\rrbracket  \\
    &  & R\llbracket X_1\rrbracket .\ar[rru]_{\evv_{F_m}} &  &  }
\end{equation*}
\noindent
\emph{Proof of Claim.} Induction on $m$. For $m=2$ the diagram above
is the $F$-iterativity diagram (Definition \ref{defefiter}). Assume that $m\geqslant 2$ and that the diagram above is commutative. Since
$$\evv_{(X_1,\ldots,X_{m-1},F(X_m,X_{m+1}))}\circ \evv_{F_m} = \evv_{F_{m+1}},$$
the commutativity of the corresponding diagram for $m+1$ follows from the commutativity of the following diagram
\begin{equation*}
  \xymatrix{R\llbracket X_1,\ldots,X_{m-1}\rrbracket \ar[rrrr]^{E_m} \ar[d]^{E_m}&   &  & & R\llbracket X_1,\ldots,X_m\rrbracket \ar[d]^{E_{m+1}} \\
  R\llbracket X_1,\ldots,X_{m}\rrbracket \ar[rrrr]^{\evv_{(X_1,\ldots,X_{m-1},F(X_m,X_{m+1}))}\ \ \ } &   &  & & R\llbracket X_1,\ldots,X_{m+1}\rrbracket .}
\end{equation*}
The commutativity of this last diagram diagram follows from the application of the functor
$$S\mapsto S\llbracket X_1,\ldots,X_{m-1}\rrbracket$$
to the $F$-iterativity diagram after setting $X=X_{m+1}$ and $Y=Z=X_m$.
\end{proof}
\begin{remark}\label{gapower}
\begin{enumerate}
\item Note that for a $\ga$-derivation $\partial$, Proposition \ref{composing} immediately implies that $\partial^{(p)}=(\id,0,0,\ldots)$.

\item If $\partial'$ is an HS-derivation such that $(\partial')_X=E_{(m)}$, then by Remark \ref{pullpusht} and Proposition \ref{composing}, $\partial'$ is $F$-iterative as well (since $[m]_{F}$ is an endomorphism of the formal group law $F$).
\end{enumerate}
\end{remark}
\begin{prop}\label{pthpower}
If $\partial$ is an $F$-derivation, then there is $W\in Xk\llbracket X\rrbracket$ such that the following diagram commutes
\begin{equation*}
  \xymatrix{
R  \ar[rr]^{\partial_X^{(p)}} \ar[rd]_{\partial_X} &  & R\llbracket X\rrbracket\\
 &  R\llbracket X\rrbracket .\ar[ru]_{\evv_W} & }
\end{equation*}
\end{prop}
\begin{proof}
From the displayed formula in the proof of Corollary \ref{ppowerder} and from Proposition \ref{composing} we get the following commutative diagram.
\begin{equation*}
  \xymatrix{
R\llbracket X\rrbracket  \ar[rr]^{\ev_{X^p}} &  & R\llbracket X^p\rrbracket \ar[rrd]^{\iota} & & \\
R \ar[u]^{\partial_X^{(p)}} \ar[rru]^{\partial_{X^p}^{(p)}} \ar[rrrr]^{E_{(p)}} \ar[rrd]_{\partial_X} & & & & R\llbracket X\rrbracket  \\
 &  & R\llbracket X\rrbracket \ar[rru]_{\ev_{[p]_F}} & & }
\end{equation*}
\noindent
By \cite[Example IV.7.1]{Si} (or Proposition on page 29 of \cite{Demazure}), there is $V\in Xk\llbracket X\rrbracket$ such that $[p]_{F}(X)=V(X^p)$. Thus we have the second commutative diagram (understanding now $\ev_{X^p}$ as a function from $R\llbracket X\rrbracket$ to itself).
\begin{equation*}
  \xymatrix{
R\llbracket X\rrbracket  \ar[rrd]^{\ev_{X^p}} &  & \\
R \ar[u]^{\partial_X^{(p)}} \ar[rr]^{E_{(p)}} \ar[d]_{\partial_X} & & R\llbracket X\rrbracket  \\
R\llbracket X\rrbracket \ar[rr]^{\ev_{X^p}} & & R\llbracket X\rrbracket \ar[u]_{\ev_V}}
\end{equation*}
\noindent
We define the power series $W$ as $V^{\fr^{-1}}$. We clearly have $V(X^p)=W^p$ and we get the third commutative diagram.
\begin{equation*}
  \xymatrix{
R \ar[rr]^{\partial_X^{(p)}} \ar[d]_{\partial_X} & & R\llbracket X\rrbracket \ar[rrd]^{\ev_{X^p}} & &  \\
R\llbracket X\rrbracket \ar[rr]^{\ev_{W}} & & R\llbracket X\rrbracket \ar[rr]^{\ev_{X^p}} & &  R\llbracket X\rrbracket }
\end{equation*}
\noindent
Since $\ev_{X^p}$ is a monomorphism, the result follows.
\end{proof}
\begin{remark}\label{gmpower}
\begin{enumerate}
\item The power series $W$ from the statement of Proposition \ref{pthpower} coincides with $(V_F)^{\fr^{-1}}$, where $V_F$ is the power series corresponding to the \emph{Verschiebung morphism} (see e.g. page 28 of \cite{Demazure}).

\item Note that for $F=X_1+X_2+X_1X_2$, the proposition above immediately implies that $\partial^{(p)}=\partial$, since in this case $W=V=X$ (it also follows from the description of the restricted Lie algebra action in \cite[page 129]{Tyc}).

\item For (additively) iterative HS-derivations if $i,j<p^m$ and $i+j\geqslant p^m$ then $\partial_{i}\circ \partial_{j}=0$. It is not true for other types of iterativity, e.g. $\partial_1\circ \partial_1=\partial_1$ for $p=2$ and a multiplicatively iterative HS-derivation $\partial=(\partial_i)_{i\in \Nn}$.

\item Proposition \ref{pthpower} remains true if we replace $F$-derivations with $f$-derivations for an $m$-truncated group law $f$, just by replacing the ring $R\llbracket X\rrbracket$ with $R[X]/(X^{p^m})$ and using \cite[Proposition 1.4]{Manin} instead of \cite[Example IV.7.1]{Si}.
\end{enumerate}
\end{remark}
Let us fix $q=p^m$. If $\partial=(\partial)_{i<q}$ is an $f$-derivation, then in general it is difficult to give formulas for $\partial_j\circ \partial_i$. We show below that the general $f$-iterativity rule resembles the standard (additive) one up to the lower order terms.
\begin{lemma}\label{approx}
Let $\partial$ be an $f$-derivation on $R$. For every positive integers $i,j$ such that $i+j<q$ we have
\begin{equation*}
 \partial_j\circ \partial_i={i+j\choose i}\partial_{i+j}+\mathcal{O}(\partial_{<i+j}),
\end{equation*}
 where $\mathcal{O}(\partial_{<i+j})$ is a
$k$-linear combination of the maps $\partial_1,\ldots,\partial_{i+j-1}$.
\end{lemma}
\begin{proof}
It follows from the $f$-iterativity diagram (Definition \ref{deffitertrun}) that for any $r\in R$ we have:
$$\sum_{i,j<q}\partial_j(\partial_i(r))v^iw^j=\sum_{n<q}\partial_n(r)f(v,w)^n.$$
Since $f$ is a truncated group law, $f=v+w+s$ for some $s\in k[v,w]$ divisible by $vw$.
Therefore we get
$$\partial_j(\partial_i(x))={i+j\choose i}\partial_{i+j}(x)+\mathcal{O}(\partial_{<i+j})(x)$$
proving the required equality.
\end{proof}
\begin{remark}\label{constcoinc}
\begin{enumerate}
\item Since for a formal group law $F$, $F[m]$ is an $m$-truncated group law, the above lemma is also true for $F$-derivations (with no restrictions for $i,j$).

\item Let $\partial$ be an $f$-derivation on $R$. Using Lemma \ref{approx}, we can find $\alpha\in k$ such that
$$\partial_1\circ \partial_1=\partial_2+\alpha\partial_1.$$
For each $0<n<p-1$, we inductively obtain that $\partial_n$ is a $k$-linear combination of non-zero compositional powers of $\partial_1$. Therefore we have
$$\ker(\partial_1)=\ker(\partial_1)\cap \ldots \cap \ker(\partial_{p-1}).$$
\end{enumerate}
\end{remark}
\begin{example}\label{multiter}
Assume that $\partial$ is a $\gm$-derivation and $k,l\in \Nn$. For any $r\in R$ we have
\begin{IEEEeqnarray*}{rCl}
\sum_{i,j}\partial_j(\partial_i(r))X^iY^j
&=&
\sum_n\partial_n(r)(X+Y+XY)^n
\\ &=&
\sum_n\partial_n(r)
\sum\limits_{a+b\leqslant n} \frac{n!}{a!b!(n-a-b)!}X^aY^b(XY)^{n-a-b}
\\ &=&
\sum_{i,j}\left(\sum\limits_{l=0}^{\min(i,j)} \frac{(i+j-l)!}{(j-l)!l!(i-l)!}\partial_{i+j-l}(r)\right)X^iY^j.
\end{IEEEeqnarray*}
Therefore we get the following multiplicative iterativity rule
\begin{equation*}
\partial_j\circ \partial_i=\sum\limits_{n=\max(i,j)}^{i+j}\frac{n!}{(n-i)!(n-j)!(i+j-n)!}\partial_{n},
\end{equation*}
which recovers the formula from \cite[Def. 11]{cr98}.
\end{example}
\noindent
The following lemma will help to find ``canonical elements'' for certain HS-derivations.
\begin{lemma}\label{ppowers}
Let $\partial, \partial^{\prime}$ be $f$-iterative derivations on $R$.
 If for every $k<m$ we have $\partial_{p^k}=\partial_{p^k}^{\prime}$, then for any $n<q$ we have $\partial_n=\partial^{\prime}_n$.
\end{lemma}
\begin{proof}
Induction on $n$. Take $n<q$ and assume that for all $k<n$ we have $\partial_k=\partial'_k$. Let $k<m$ be biggest such that $p^k\leqslant n$. By the assumption we can assume that $n=p^k+l$ for some positive integer $l$. Then $p$ does not divide ${n\choose l}$. Take any $r\in R$, by Lemma \ref{approx}, we have
$$ \partial_{p^k}(\partial_l(r))={n\choose l}\partial_{n}(r)+\mathcal{O}(\partial_{<n})(r),$$
$$ \partial'_{p^k}(\partial'_l(r))={n\choose l}\partial'_{n}(r)+\mathcal{O}(\partial'_{<n})(r).$$
By the induction assumption we get $\partial_{n}(r)=\partial'_{n}(r)$.
\end{proof}
\begin{cor}\label{pvanish}
Let $\partial$ be an $f$-iterative derivation on $R$ and $r\in R$. If
 for every $k<m$ we have $\partial_{p^k}(r)=0$, then for any $n<q$ we have $\partial_n(r)=0$.
\end{cor}
\begin{remark}
\begin{enumerate}
\item Similarly as before, the appropriate versions of \ref{ppowers} and \ref{pvanish} for $F$-derivations are also true.
\item Similar formulas to the ones appearing in the proof of Lemma \ref{ppowers} (in the case of the additive iterativity rule) can be found at the top of page 175 of \cite{ResMat}.
\end{enumerate}
\end{remark}

\subsection{Fields of constants}
In this subsection we generalize a result of Matsumura \cite[Theorem 27.3]{mat} saying that the field of constants $C$ of a non-zero derivation $\partial$ on a field $K$ such that $\partial^{(p)}=0$ is largest possible i.e. $[K:C]=p$. If $\partial$ is a truncated HS-derivation or an HS-derivation, then we define its \emph{field of constants} as the intersection of $\ker(\partial_i)$ for all $i\neq 0$. By Remark \ref{oneiterat}, having a derivation $\partial$ such that $\partial^{(p)}=0$ is equivalent to having a $\ga[1]$-derivation $(\partial_i)_{i<p}$ and by Remark \ref{constcoinc}(2), both fields of constants coincide. In Proposition \ref{degree}, we generalize Matsumura's result  from $\ga[1]$-derivations to $f$-derivations, where $f$ is an arbitrary truncated group law. Such a generalization will be necessary for integrating truncated HS-derivations.
\\
\\
We will need a slight generalization of a result from \cite{mat} (recall that $q=p^m$).
\begin{lemma}\label{indepenendent}
Let $\partial$ be a non-zero derivation on a field $K$ of characteristic $p$. Then the functions $\id_K^q,\partial^{q},\ldots,(\partial^{(p-1)})^{q}$ are linearly independent over $K$.
\end{lemma}
\begin{proof} The case of $m=0$ is exactly \cite[Theorem 25.4]{mat}. For the general case it is enough to notice that the $K$-linear independence of
$\id_K,\partial,\ldots,\partial^{(p-1)}$ implies the $K^{q^{-1}}$-linear independence of $\iota ,\iota \circ\partial,\ldots,\iota \circ\partial^{(p-1)}$, where $\iota:K\to K^{q^{-1}}$ is the inclusion map. Applying the $m$-th power of the Frobenius map we get the independence of  $\id_K^q,\partial^{q},\ldots,(\partial^{(p-1)})^{q}$ over $K$.
\end{proof}
\noindent
The result below is a multiplicative version of a part of \cite[Theorem 27.3(ii)]{mat}.
\begin{lemma}\label{dxisx}
Let $\partial$ be a non-zero derivation on a field $K$ such that $\partial^{(p)}=\partial$. Then there is a non-zero $x\in K$ such that $\partial(x)=x$.
\end{lemma}
\begin{proof}
 We argue by contradiction. Suppose that $\partial(x)\neq x$ for all non-zero $x\in K$.
It means that the map $\partial-\id_K$ is one-to-one.
Notice that
\begin{equation*}
 (\partial-\id_{K})\circ(\id_{K}+\partial+\partial^{(2)}+\ldots+\partial^{(p-1)})=\partial^{(p)}-\id_{K}=
\partial-\id_{K}.
\end{equation*}
From the injectivity of $\partial-\id_{K}$, we obtain
\begin{equation*}
 \id_{K}+\partial+\partial^{(2)}+\ldots+\partial^{(p-1)}=\id_{K},
\end{equation*}
which contradicts \cite[Theorem 25.4]{mat} (see Lemma \ref{indepenendent} above for $q=1$).
\end{proof}
\begin{remark}
Such an element $x$ is also given by \cite[Lemma 3.5(3)]{Tyc}.
\end{remark}
\noindent
Before the next result, we need a very general lemma, which is not particularly related to derivations.
\begin{lemma}\label{general}
Let $\partial:R\to R$ be a function.
\begin{enumerate}
\item  We have
$$\partial^{(p-1)}+\sum\limits_{l=1}^{p-1}\sum\limits_{i=1}^{p-1}l^{p-1-i}\partial^{(i)}=0.$$
\item Let $l\in \Nn$ and assume that $\partial$ is additive and $\partial^{(p)}=\partial$. Then
$$  \partial\Big(\sum\limits_{i=1}^{p-1}l^{p-1-i}\partial^{(i)}\Big)=
	  l\Big(\sum\limits_{i=1}^{p-1}l^{p-1-i}\partial^{(i)}\Big).$$
\end{enumerate}
\end{lemma}
\begin{proof}
Note that for any $\alpha\in \Ff_p$, since $\alpha^p=\alpha$, we get
\begin{equation}
\sum_{r=1}^{p-1}\alpha^r=\alpha(\sum_{r=1}^{p-1}\alpha^r).\tag{$*$}
\end{equation}
Let $\xi$ be a generator of $\Ff_p^*$ and $i\in \{1,\ldots,p-2\}$. Then we have
\begin{equation}
 \sum\limits_{l=1}^{p-1}l^{i}=\sum\limits_{r=1}^{p-1}(\xi^r)^i=\sum\limits_{r=1}^{p-1}(\xi^i)^r=0,\tag{$**$}
\end{equation}
where the last equality follows from $(*)$ (for $\alpha=\xi^i$), since $\xi^i\neq 1$.
\\
Using $(**)$ we get
\begin{IEEEeqnarray*}{rCl}
\partial^{(p-1)}+\sum\limits_{l=1}^{p-1}\sum\limits_{i=1}^{p-1}l^{p-1-i}\partial^{(i)}
&=&
\partial^{(p-1)}+\sum\limits_{i=1}^{p-1}(\sum\limits_{l=1}^{p-1}l^{p-1-i})\partial^{(i)}
\\ &=&
\partial^{(p-1)}+\sum_{l=1}^{p-1}l^{0}\partial^{(p-1)}
\\ &=&  0.
\end{IEEEeqnarray*}
The proof of $(2)$ is an easy computation which is similar to the one needed to obtain the equality $(*)$ above.
\end{proof}
\noindent
We prove now the main result of this subsection. Recall that $f$ is a fixed $m$-truncated group law.
\begin{prop}\label{degree}
Let $K$ be a field and $\partial$ be an $f$-derivation on $K$ such that $\partial_1$ is non-zero. Let $C$ be the constant field of $\partial$. Then $[K:C]=p^m$.
\end{prop}
\begin{proof}
First we show inductively that without loss of generality we can assume that $m=1$. Assume that $m>1$ and let $C'=\ker(\partial_1)$. By Lemma \ref{hsfrob}(4) and Corollary \ref{pvanish}, $\partial'=(\partial_{pi}|_{C'})_{i<p^{m-1}}$ is an $f'[m-1]$-derivation on $C'$, where $f'=f^{\fr}$.
\\
By Lemma \ref{approx} and Corollary \ref{pvanish}, the constant field of $\partial'$ coincides with $C$, so we are done by the inductive assumption and by the $m=1$ case. Assume  that $m=1$.
\\
By Remark \ref{constcoinc}(2), we have $C=\ker(\partial_1)$. By Remark \ref{gmpower}(4), there is $t\in vk[v]$ such that
$$\partial^{(p)}_v=\ev_t\circ \partial_v.$$
Hence we have $\partial_1^{(p)}=c\partial_1$ for $c\in k$ such that $t-cv\in v^2k[v]$ (i.e. $t=cv+\ldots$). The case  $\partial_1^{(p)}=0$ is treated in \cite[Theorem 27.3]{mat}, so we can assume that $c\neq 0$. We assume first that $c=1$ (which corresponds to the multiplicative case) and will treat the general case at the end of the proof.
\\
Let us take an arbitrary $a\in K$ and let
 $x\in K\setminus \{0\}$ be such that $\partial(x)=x$ (see Lemma \ref{dxisx}). By Lemma \ref{general}(2), we have
$$a=a-\partial^{(p-1)}(a)-\sum\limits_{k=1}^{p-1}\frac{\alpha_k}{x^k}x^{k},$$
where $\alpha_k=\sum\limits_{i=1}^{p-1}k^{p-1-i}\partial^{(i)}(a)$. By Lemma \ref{general}(1), we have $\partial(\alpha_k)=k\alpha_k$. Since $\partial(x)=x$, it is easy to see that
$$\partial\left(\frac{\alpha_k}{x^k}\right)=0.$$
Therefore $K$ is spanned over $C$ by $\{1,x,\ldots,x^{p-1}\}$, so $[K:C]=p$.
\\
Let us come back now to the general case when $\partial_1^{(p)}=c\partial_1$ for an arbitrary $c\in k\setminus \{0\}$. Note that for any $d\in C$ and $\partial_d:=d\partial$ we have
$$\partial_d^{(p)}=d^p\partial^{(p)}=d^pc\partial=d^{p-1}c\partial_d.$$
(The first equality is easy to see, it is also a special case of the Hochschild formula, see \cite[Theorem 25.5]{mat}.)
If $d\neq 0$, then the constants of $\partial_d$ coincide with the constants of $\partial$, so we can replace $\partial$ with $\partial_d$. Therefore we are done if there is $d\in C$ such that $d^{p-1}=c^{-1}$. Let $K'$ denote the separable closure of $K$. The derivation $\partial$ uniquely extends to a derivation $\partial'$ on $K'$. Let $C'$ denote the field of constants of $\partial'$. From the uniqueness of extensions of derivations, we get  $(\partial')^{(p)}=c\partial'$. Clearly, there is $d\in K'$ such that $d^{p-1}=c^{-1}$. Therefore $[K':C']=p$. Since $C'$ is linearly disjoint from $K$ over $C$ (see \cite[Corollary 1 p. 87]{kol1}), we get $[K:C]=p$.
\end{proof}

\section{Expanding HS-derivations}\label{secexpanding}
\noindent
In this section we generalize the result of Matsumura \cite[Theorem 7]{Mats1} about strong integrability of certain derivations. Throughout this section $k\subseteq M\subseteq K$ is a tower of fields such that $k$ is perfect (of characteristic $p>0$) and the extension $M\subseteq K$ is separable (not necessarily algebraic). We fix $m>0$ and set $q:=p^m$.
\\
\\
Our generalization is two-fold:
\begin{enumerate}
\item From an $M$-derivation $D$ on $K$ such that $D^{(p)}=0$ (equivalent to a $\ga[1]$-derivation, see Remark \ref{oneiterat}) to any $\ga[m]$-derivation (Theorem \ref{mainadditive}).

\item Analogue of (1) for $\gm[m]$-derivations (Theorem \ref{mainmultiplicative}).
\end{enumerate}
\noindent
It should be mentioned that Proposition \ref{caninacala}, which is the main technical point needed for $(1)$ above, was essentially obtained by Ziegler in \cite[Theorem 1]{Zieg3}.
\\
\\
It is natural to ask whether similar results can be obtained for any formal group law.
\begin{question}\label{mainquestion}
Let $F$ be a formal group law over $k$ and $\partial$ be an $F[m]$-derivation on $K$. Does $\partial$ expand to an $F$-derivation?
\end{question}
\noindent
Unfortunately, we do not know the answer to this question. In \cite{HK2}, we give some evidence why the answer may be negative.
\\
\\
For the notion of a $p$-basis of an extension of fields of characteristic $p>0$, the reader may consult \cite[p. 202]{mat}. The proof of \cite[Theorem 27.3(ii)]{mat} gives the following.
\begin{lemma}\label{pbasis0}
Let $M\subseteq L\subseteq K$ be a tower of fields ($M\subseteq K$ separable), $[K:L]=p$, $L\subseteq K$ is purely inseparable and $a\in K\setminus L$. Then we have:
\begin{enumerate}
\item There is $B_0\subseteq L$ such that $B_0\cup \{a^p\}$ is a $p$-basis of $L$ over $M$.

\item For any $B_0\subseteq L$ such that $B_0\cup \{a^p\}$ is a $p$-basis of $L$ over $M$, the set $B_0\cup \{a\}$ is a $p$-basis of $K$ over $M$.
\end{enumerate}
\end{lemma}
\noindent
We will need a generalization of the second part of the lemma above.
\begin{lemma}\label{pbasis}
Let $m\in \Nn$ and $M\subseteq K_{m-1}\subseteq \ldots K_0\subseteq K$
be a tower of fields ($M\subseteq K$ separable) such that $K_{m-1}\subseteq K$ is purely inseparable and
$$[K:K_0]=[K_0:K_1]=\ldots=[K_{m-2}:K_{m-1}]=p.$$
Let $x\in K$ be such that for all $i\in \{0,\ldots,m-1\}$, we have $x^{p^i}\notin K_i$. Then there is $B_0\subseteq K_{m-1}$ such that $B_0\cup \{x\}$ is a $p$-basis of $K$ over $M$.
\end{lemma}
\begin{proof}
Clearly, for all $i\in \{0,\ldots,m-1\}$, the extension $M\subseteq K_i$ is separable. Using Lemma \ref{pbasis0}(1) for $a=x^{p^{m-1}}$ and $M\subseteq K_{m-1}\subseteq K_{m-2}$, we get $B_0\subseteq K_{m-1}$ such that $B_0\cup \{x^{p^{m}}\}$ is a $p$-basis of $K$ over $M$. Using Lemma \ref{pbasis0}(2) and the downward induction on $i$, we get that $B_0\cup \{x\}$ is a $p$-basis of $K$ over $M$.
\end{proof}
\noindent
For the remainder of this section we consider only the additive and the multiplicative laws, so we can assume that $k$ is the prime field $\Ff_p$.

\subsection{Additive case}\label{secadd}
We will need one fact about derivations.
\begin{lemma}\label{imisker}
Assume that $\partial$ is a non-zero derivation on $K$ such that $\partial^{(p)}=0$. Then
$$\im(\partial)=\ker(\partial^{(p-1)}).$$
\end{lemma}
\begin{proof} Let $C=\ker(\partial)$. By \cite[Theorem 27.3]{mat}, (see also Proposition \ref{degree}) $[K:C]=p$.
Clearly $\im(\partial)\subseteq \ker(\partial^{(p-1)})$. It is enough to show that
$$\dim_C\im(\partial)\geqslant \dim_C\ker(\partial^{(p-1)}).$$
By \cite[Theorem 25.4]{mat} (see also Lemma \ref{indepenendent}), $\partial^{(p-1)}\neq 0$, so $\dim_C\ker(\partial^{(p-1)})\leqslant p-1$.
\\
By \cite[Theorem 27.3(ii)]{mat}, there is $x\in K$ such that $\partial(x)=1$ (namely, $x=\partial^{(i-1)}(z)/\partial^{(i)}(z)$ for $z\in K\setminus C$ and $i>0$ such that $\partial^{(i)}(z)\neq 0$ and $\partial^{(i+1)}(z)=0$). Therefore, $\partial(x^n)=nx^{n-1}$ for $n=0,\ldots,p-1$ and $1,x,x^2,\ldots,x^{p-2}$ are linearly independent over $C$. Hence $\dim_C(\im(\partial))\geqslant p-1$.
\end{proof}
\noindent
We show now that any $\ga[m]$-derivation has something in common with the canonical one.
\begin{prop}\label{caninacala}
Let $\partial$ be a $\ga[m]$-derivation on $K$ such that $\partial_1$ is non-zero. Then there is $x\in K$ such that
$$\partial_1(x)=1,\ \ \ \partial_{2}(x)=0,\ \ \ldots,\ \partial_{q-1}(x)=0.$$
\end{prop}
\begin{proof}
Induction on $m$. The case of $m=1$ is \cite[Theorem 27.3(ii)]{mat}. Assume that the proposition is true for $m$ and take a $\ga[m+1]$-derivation $\partial$ on $K$ over $k$ such that $\partial_1$ is non-zero. Let $\partial'=(\partial_i)_{i<p^{m}}$. By Lemma \ref{hsfrob}(2), $\partial'$ is a $\ga[m]$-derivation.
\\
By the inductive assumption, there is $x\in K$ such that
\begin{equation}
\partial_1(x)=1,\ \  \partial_2(x)=0,\ \ldots\ ,\partial_{p^{m}-1}(x)=0.\tag{$*$}
\end{equation}
Let $C$ be the field of constants of $\partial$ and $C'$ the field of constants of $\partial'$. By higher Leibniz rules, $\partial_{p^m}$ is a $C$-derivation on $C'$. It is easy to see that $\partial_{p^m}$ is non-zero on $C'$, since $x^{p^m}\in C'$ and
$$\partial_{p^m}(x^{p^m})=\partial_1(x)^{p^m}\neq 0.$$
By Remark \ref{gapower}(1), $\partial_{p^m}^{(p)}=0$, so by Lemma \ref{imisker} we have
\begin{equation}
\partial_{p^m}(C')=\ker(\partial_{p^m}^{(p-1)})\cap C'.\tag{$**$}
\end{equation}
Since $\partial_{p^m}^{(p)}=0$, we have $\partial_{p^m}(x)\in \ker(\partial_{p^m}^{(p-1)})$. Since $\partial_{p^m}$ commutes with $\partial_1,\ \ldots\ ,\partial_{p^{m-1}}$, we get by $(*)$ that $\partial_{p^m}(x)\in C'$. By $(**)$, there is $y\in C'$ such that
$$\partial_{p^m}(y)=\partial_{p^m}(x).$$
Let $z:=y-x$. It is clear that
$$
\partial_1(z)=1,\ \  \partial_2(z)=0,\ \ldots\ ,\partial_{p^{m}}(z)=0.
$$
By the iterativity rule, for any $i<(p-1)p^m$ we also have $\partial_{p^{m}+i}(z)=0$.
\end{proof}
\noindent
We will call an element $x\in K$ satisfying the conclusion of Proposition \ref{caninacala}, a \emph{canonical} element for a $\ga[m]$-derivation $\partial$.
\begin{remark}
Ziegler in \cite{Zieg3} and \cite{Zieg2} defines the notion of a \emph{canonical $p$-basis} which in our ``ordinary'' (i.e. one HS-derivation) case reduces to the notion of a canonical element. For the sake of clarity we restrain ourselves from considering the case of several HS-derivations, see Section \ref{multi}.
\end{remark}
\noindent
We can prove now our additive integrability theorem. This theorem is implicit in work of Ziegler, see \cite{Zieg3} and \cite{Zieg2}. Let us recall that $M\subseteq K$ is a separable (not necessarily algebraic) extension of fields.
\begin{theorem}\label{mainadditive}
Let $\partial=(\partial_i)_{i<p^m}$ be a $\ga[m]$-derivation on $K$ over $M$. Then $\partial$ can be expanded to a $\ga$-derivation on $K$ over $M$.
\end{theorem}
\begin{proof}
Let us assume first that $\partial_1$ is non-zero. The proof in this case is similar to the proof of \cite[Theorem 7]{Mats1}.
Take a canonical element $x\in K$ from Proposition \ref{caninacala}. For $i\in \{0,\dots,m-1\}$ let us define
$$K_i:=\bigcap_{j<p^{i+1}}\ker(\partial_i).$$
By Proposition \ref{degree}, for each $i\in \{0,\dots,m-1\}$, we have $[K:K_i]=p^{i+1}$ and
$$\partial_{p^{i}}(x^{p^{i}})=\partial_1(x)^{p^{i}}=1\neq 0,$$
so $x^{p^{i}}\notin K_i$. Hence $M,K_{m-1},\ldots,K_0,K,x$ satisfy the assumptions of Lemma \ref{pbasis}, therefore there is $B_0\subset K_{m-1}$ such that $B:=B_0\cup \{x\}$ is a $p$-basis of $K$ over $M$. By \cite[Theorem 26.8]{mat} $x$ is transcendental over $M(B_0)$, so we can define a canonical $\ga$-derivation on $M(B)$ over $M(B_0)$ (see Example \ref{gaitex}). By \cite[Theorem 26.8]{mat} again, $M(B)\subseteq K$ is \'{e}tale, so by \cite[Theorem 27.2]{mat} our canonical $\ga$-derivations uniquely extends to a $\ga$-derivation $D$ on $K$ over $M$. Since for any $i<q$, $\partial_i$ coincides with $D_i$ on $B$ (therefore on $M(B)$ as well), it is easy to see that they coincide on $K$.
\\
\\
Let us now consider the case $\partial_1=0$. Obviously, we can assume that $\partial_i$ is non-zero for some $0<i<p^m$. By Lemma \ref{approx} and Corollary \ref{pvanish}, there is $j<m$ such that $\partial_1=\ldots=\partial_{p^j-1}=0$ but $\partial_{p^{j}}$ is non-zero. Let $l:=m-j$ and $\partial'=(\partial_{ip^j})_{i<p^l}$. By Lemma \ref{hsfrob}(4), $\partial'$ is a $\ga[l]$-derivation (since $(X+Y)^{\fr^{-j}}$ clearly coincides with $X+Y$). Since $\partial_1'$ is non-zero, by the first part of the proof $\partial'$ expands to a $\ga$-derivation $D'=(D_i')_{i\in \Nn}$. Define $D=(D_i)_{i\in \Nn}$, where $D_i=D'_{i/p^j}$ for $i$ divisible by $p^j$ and $D_i=0$ otherwise. Using the map
$$\ev_{X^{p^l}}:K\llbracket X\rrbracket\to K\llbracket X\rrbracket,$$
which gives a morphism from $\ga$ to $\ga$ and Lemma \ref{pullpushf}, we see that $D$ is a $\ga$-derivation. By the construction, $D$ expands $\partial$.
\end{proof}

\subsection{Multiplicative case}\label{secmult}
In this subsection we prove an analogue of Theorem \ref{mainadditive} for multiplicatively iterative derivations. The general scheme of the proof is the same as in the additive case: the main point is to find a canonical element (in the appropriate, multiplicative sense) for a truncated multiplicatively iterative HS-derivation, this is done in Proposition \ref{caninacalm}.
\\
\\
The lemma below does not require any iterativity assumption.
\begin{lemma}\label{lemma1}
Let $\partial$ be an $m$-truncated HS-derivation on a ring $R$. For any $i<m$,
$j\in \Nn$, $x\in R$ and $y\in R$ we have
\begin{equation}
\partial_{p^{i}}^{(j)}(xy^{p^i})=\sum_{l=0}^{j}{j \choose l}\partial^{(j-l)}_{p^i}(x)  \cdot\partial_1^{(l)}(y)^{p^i}.\tag{$*$}
\end{equation}
\end{lemma}

\begin{proof}
We prove $(*)$ by induction on $j$. The case of $j=0$ is clear. Let us assume that $(*)$ above holds. Then we have
\begin{IEEEeqnarray*}{rCl}
\partial_{p^i}^{(j+1)}(xy^{p^i})
&=&
\sum_{l=0}^{j}{j \choose l}\partial_{p^i}\Big(\partial^{(j-l)}_{p^i}(x)  \cdot\partial_1^{(l)}(y)^{p^i}\Big)
\\  &=& \sum_{l=0}^{j}{j \choose l}\Big(\partial_{p^i}(\partial^{(j-l)}_{p^i}(x))  \cdot\partial_1^{(l)}(y)^{p^i}
+\partial^{(j-l)}_{p^i}(x)  \cdot \partial_{p^i}(\partial_1^{(l)}(y)^{p^i})\Big)
\\  &=& \sum_{l=0}^{j}{j \choose l}\Big(\partial^{(j+1-l)}_{p^i}(x)  \cdot\partial_1^{(l)}(y)^{p^i}
+\partial^{(j-l)}_{p^i}(x)  \cdot\partial_1^{(l+1)}(y)^{p^i}\Big)
\\  &=& \sum_{l=0}^{j+1}{j+1 \choose l}\partial^{(j+1-l)}_{p^i}(x)  \cdot\partial_1^{(l)}(y)^{p^i},
\end{IEEEeqnarray*}
where the second equality holds, since for every $0<s<p^i$ and $a\in R$ we have $\partial_s(a^{p^i})=0$.
\end{proof}
\begin{lemma}\label{lemma2}
Let $\partial$ be a $\gm[m]$-derivation on a ring $R$ and $i\in \{0,\ldots,m-2\}$. Assume that $x\in R$ satisfies
$$\partial_1(x)=x,\ \ \partial_p(x)=0,\ \ldots \ ,\partial_{p^i}(x)=0.$$
Take any $y\in R$ and let $x':=y^{p^{i+1}}x-\partial_{p^{i+1}}^{(p-1)}(y^{p^{i+1}}x)$. Then we have:
$$\partial_1(x')=x',\ \ \partial_p(x')=0,\ \ldots \ ,\partial_{p^{i+1}}(x')=0.$$
\end{lemma}
\begin{proof}
We compute
\begin{IEEEeqnarray*}{rCl}
\partial_1(x')
&=&
y^{p^{i+1}}\partial_1(x)-\partial_{p^{i+1}}^{(p-1)}(y^{p^{i+1}}\partial_1(x))
\\ &=&
y^{p^{i+1}}x-\partial_{p^{i+1}}^{(p-1)}(y^{p^{i+1}}x)
\\ &=&
x';
\end{IEEEeqnarray*}
$$\partial_{p^{i+1}}(x')=\partial_{p^{i+1}}(y^{p^{i+1}}x)-\partial_{p^{i+1}}^{(p)}(y^{p^{i+1}}x)=0.$$
Take any $j\in \{1,\ldots,i\}$. We have
\begin{equation}
\partial_{p^j}(x')=\partial_{p^j}(y^{p^{i+1}}x)-\partial_{p^j}(\partial_{p^{i+1}}^{(p-1)}(y^{p^{i+1}}x)).\tag{$*$}
\end{equation}
\noindent
By the choice of $j$ we have $\partial_{p^j}(y^{p^{i+1}})=0$ and $\partial_{p^j}(x)=0$, hence the first term in $(*)$ vanishes. The second one vanishes for the same reason, since $\partial_{p^j}$ commutes with $\partial_{p^{i+1}}$.
\end{proof}
\noindent
We show below that canonical elements for $\gm[m]$-derivations exist as well.
\begin{prop}\label{caninacalm}
Assume that  $\partial$ is a $\gm[m]$-derivation such that $\partial_1$ is non-zero. Then there is $x\in K$ such that
$$\partial_1(x)=x+1\neq 0,\ \ \partial_{2}(x)=0,\ \ldots\ , \partial_{q-1}(x)=0.$$
\end{prop}
\begin{proof}
First, we inductively construct a sequence $x_0,\ldots,x_{m-1}\in K\setminus\{0\}$ such that for all $i\in \{0,\ldots,m-1\}$ we have
\begin{equation}
\partial_1(x_i)=x_i,\ \ \partial_{p}(x_i)=0,\ \ldots\ , \partial_{p^i}(x_i)=0.\tag{$*_i$}
\end{equation}
\noindent
Using Lemma \ref{dxisx}, we get a non-zero $x_0\in K$ such that $\partial_1(x_0)=x_0$. Assume that $i<m-1$ and that we have a non-zero $x_i\in K$ satisfying $(*_i)$.
\\
By Lemma \ref{lemma1} for all $y\in K$ we have (setting $q:=p^{i+1}$),
\begin{IEEEeqnarray*}{rCl}
y^{p^{i+1}}x_i-\partial_{p^{i+1}}^{(p-1)}(y^{p^{i+1}}x_i)
&=&
y^{p^{i+1}}x_i-\sum_{l=0}^{p-1}{p-1 \choose l}\partial^{(p-1-l)}_{p^{i+1}}(x_i)  \cdot\partial_1^{(l)}(y)^{p^{i+1}}\\
  &=& \Big(\sum_{l=0}^{p-2} \alpha_i \cdot \partial_1^{(l)}(y)^{q}\Big) + x_i\partial_1^{(p-1)}(y)^{q},
\end{IEEEeqnarray*}
\noindent
for some $\alpha_0,\ldots,\alpha_{p-2}\in K$. Since $x_i\neq 0$, by Lemma \ref{indepenendent} there is $y_0\in K$ such that
$$x_{i+1}:=y_0^{p^{i+1}}x_i-\partial_{p^{i+1}}^{(p-1)}(y_0^{p^{i+1}}x_i)$$
is non-zero. By Lemma \ref{lemma2}, the element  $x_{i+1}$ satisfies $(*_{i+1})$.
\\
Let $x:=x_{m-1}-1$. Then we have
$$\partial_1(x)=\partial_1(x_{m-1})=x_{m-1}=x+1\neq 0.$$
Clearly $\partial_{p^i}(x)=0$ for $i=1,\ldots,m-1$. Since $\partial_1(x)\neq 0$, $x$ is necessarily transcendental over $\Ff_p$. By Example \ref{canderm}, there is a canonical $\gm$-derivation $\partial'$ on $\Ff_p[x]$. By the definition of the canonical $\gm$-derivation, we have $\partial'_1(x)=x+1$ and $\partial'_j(x)=0$ for $j>0$ (since the multiplicative formal group law is given by the following power series $X+(X+1)Y$). Hence $\partial_{p^i}(x)=\partial'_{p^i}(x)$ for $i=1,\ldots,m-1$. By Lemma \ref{ppowers}, $\partial_{n}(x)=\partial'_{n}(x)$ for $n<p^m-1$, hence $\partial_n(x)=0$ for $n\in \{2,\ldots,p^m-1\}$.
\end{proof}
\noindent
We have now all the ingredients to prove a multiplicative version of Theorem \ref{mainadditive}. Since the proof is exactly the same as the proof of Theorem \ref{mainadditive}, we will skip it.
\begin{theorem}\label{mainmultiplicative}
Let $M\subseteq K$ be a separable field extension and $\partial$ a $\gm[m]$-derivation on $K$ over $M$. Then $\partial$ can be expanded to a $\gm$-derivation on $K$ over $M$.
\end{theorem}
\begin{remark}
The case of $m=1$ is contained in the main theorem of \cite{Tyc}.
\end{remark}

\subsection{Mixed case}
For an arbitrary $c\in k$, one can consider the formal group law $F_c:=X+Y+cXY$ and $F_c$-iterative derivations (as in \cite{Tyc}). Clearly, $F_c$ specializes to the additive law (for $c=0$) and to the multiplicative law (for $c=1$). However, it is easy to see that for $c\neq 0$, we have $F_c\cong F_1$. Thus by Lemma \ref{pullpushf}, considering $F_c$-iterative derivations reduces to considering additively iterative ones (Section \ref{secadd}) and multiplicatively iterative ones (Section \ref{secmult}).
\\
For the sake of completeness, we notice that if $\partial=(\partial_i)_i$ is an $F_c$-iterative derivation, then $\partial_1^{(p)}=c^{p-1}\partial_1$.

\section{Several HS-derivations}\label{multi}
\noindent
We could have extended the results of Section \ref{secexpanding} to the case of several iterative HS-derivations. For example, the crucial notion of a canonical element would be replaced with the notion of a \emph{canonical $p$-basis} as in \cite{Zieg2}. However, we feel that such a generalization would not be a satisfactory one. For example, a tuple of $e$ commuting $\ga$-derivations is the same as a ``$\ga^e$-derivation'', see \cite[Prop. 2.20]{MS}. Therefore one should consider $G$-derivations for an arbitrary (not necessarily commutative!) algebraic group or even a formal group. Such sequences of HS-derivations are studied in \cite{HK}.
\\
\\
It may be also interesting to compare Pierce's theory of several derivations in arbitrary characteristic \cite{Pierce3} where the composition of derivations is governed by a Lie algebra $\mathfrak{g}$ with the theory of $G$-derivations for $\lie(G)=\mathfrak{g}$.

\bibliographystyle{plain}
\bibliography{harvard}

\begin{thebibliography}{10}

\bibitem{cr98}
G~Crupi, M.;~Restuccia.
\newblock Integrable derivations and formal groups in unequal characteristic.
\newblock {\em Rend. Circ. Mat. Palermo}, 47(2):169--190, 1998.

\bibitem{cr2001}
G~Crupi, M.;~Restuccia.
\newblock Iterative differentiations in rings of unequal characteristic.
\newblock {\em Boll. Unione Mat. Ital. Sez. B Artic. Ric. Mat.}, 4(3):635--646,
  2001.

\bibitem{Demazure}
M.~Demazure.
\newblock {\em Lectures on $p$-divisible groups}, volume 302 of {\em Lecture
  Notes in Mathematics}.
\newblock Springer-Verlag, 1972.

\bibitem{comm}
D.~Eisenbud.
\newblock {\em Commutative Algebra with a View Towards Algebraic Geometry}.
\newblock Springer, 1996.

\bibitem{Gil}
H.~Gillet.
\newblock Differential algebra: A scheme theory approach.
\newblock In {\em Differential Algebra and Related Topics}, pages 95--123.
  World Sci., River Edge, N.J., Newark, N.J., 2002.

\bibitem{Hasse}
D.~Hasse.
\newblock Theorie der h\"{o}heren {D}ifferentiale in einem algebraischen
  {F}unktionenk\"{o}rper mit vollkommenem {K}onstantenk\"{o}rper bei beliebiger
  {C}harakteristik.
\newblock {\em J. reine und angew. Math.}, 175:50--54, 1936.

\bibitem{Hazew}
Michiel Hazewinkel.
\newblock {\em Formal Groups and Applications}.
\newblock Academic Press, 1978.

\bibitem{Ho}
Wilfrid Hodges.
\newblock {\em A shorter model theory}.
\newblock Cambridge University Press, Cambridge, 1997.

\bibitem{HK}
Daniel Hoffmann and Piotr Kowalski.
\newblock Existentially closed fields with ${G}$-derivations.
\newblock Available on \verb"http://arxiv.org/abs/1404.7475".

\bibitem{HK2}
Daniel Hoffmann and Piotr Kowalski.
\newblock A note on integrating group scheme actions.
\newblock Available on
  \verb"http://www.math.uni.wroc.pl/~pkowa/mojeprace/integnote.pdf".

\bibitem{kol1}
E.R. Kolchin.
\newblock {\em Differential Algebra and Algebraic Groups}.
\newblock Pure and applied mathematics. Academic Press, 1973.

\bibitem{K3}
Piotr Kowalski.
\newblock Geometric axioms for existentially closed {H}asse fields.
\newblock {\em Annals of Pure and Applied Logic}, 135:286--302, 2005.

\bibitem{Manin}
Yu~I. Manin.
\newblock The theory of commutative formal groups over fields of finite
  characteristic.
\newblock {\em Russ. Math. Surv.}, 18(6):1--83, 1963.

\bibitem{Mats1}
Hideyuki Matsumura.
\newblock Integrable derivations.
\newblock {\em Nagoya Math. J.}, 87:227--245, 1982.

\bibitem{mat}
Hideyuki Matsumura.
\newblock {\em Commutative ring theory}.
\newblock Cambridge University Press, 1986.

\bibitem{MS}
Rahim Moosa and Thomas Scanlon.
\newblock Generalized {H}asse-{S}chmidt varieties and their jet spaces.
\newblock {\em Proc. Lond. Math. Soc.}, 103(2):197--234, 2011.

\bibitem{Pierce3}
D.~Pierce.
\newblock Fields with several commuting derivations, to appear in the {J}ournal
  of {S}ymbolic {L}ogic.
\newblock Available on \\
  \verb"http://mat.msgsu.edu.tr/~dpierce/Mathematics/Derivations/Revisited/".

\bibitem{ResMat}
H.~Restuccia, G.;~Matsumura.
\newblock Integrable derivations in rings of unequal characteristic.
\newblock {\em Nagoya Math. J.}, 93:173--178, 1984.

\bibitem{HS}
F.K. Schmidt and D.~Hasse.
\newblock Noch eine {B}egr\"{u}ndung der {T}heorie der h\"{o}heren
  {D}ifferentialquotienten in einem algebraischen {F}unktionenk\"{o}rper einer
  {U}nbestimmten.
\newblock {\em J. reine und angew. Math.}, 177:215–--237, 1937.

\bibitem{Si}
Joseph~H. Silverman.
\newblock {\em The Arithmetic of Elliptic Curves}.
\newblock Graduate Texts in Mathematics. Springer-Verlag, 1986.

\bibitem{Tyc}
A.~Tyc.
\newblock On ${F}$-integrable actions of the restricted {L}ie algebra of a
  formal group ${F}$ in characteristic $p > 0$.
\newblock {\em Nagoya Math. J.}, 115:125--137, 1989.

\bibitem{Water}
William~C. Waterhouse.
\newblock {\em Introduction to {A}ffine {G}roup {S}chemes}.
\newblock Springer-Verlag, 1979.

\bibitem{Zieg3}
M.~Ziegler.
\newblock Canonical $p$-bases, preprint.
\newblock Available on \\
  \verb"http://home.mathematik.uni-freiburg.de/ziegler/preprints/canonical-p-bases.pdf".

\bibitem{Zieg2}
M.~Ziegler.
\newblock Separably closed fields with {H}asse derivations.
\newblock {\em Journal of Symbolic Logic}, 68:311--318, 2003.

\end{thebibliography}

\end{document}